\documentclass[12]{amsart}
\usepackage{amsmath,amssymb}
\usepackage{amsthm,color,enumerate,comment,centernot,enumitem,url,cite}
\usepackage{graphicx,relsize,bm}
\usepackage{mathtools}
\usepackage{array}
\usepackage{enumitem}
\setenumerate[0]{label=(\alph*)}

\makeatletter
\newcommand{\tpmod}[1]{{\@displayfalse\pmod{#1}}}
\makeatother

\newcommand{\ord}{\operatorname{ord}}

\newtheorem{thm}{Theorem}[section]
\newtheorem{lemma}[thm]{Lemma}

\newtheorem{prop}[thm]{Proposition}

\theoremstyle{remark}

\theoremstyle{definition}
    \newtheorem{defn}[thm]{Definition}

\newtheorem{rem}[thm]{Remark}

\newcommand{\abs}[1]{\left|{#1}\right|}

\def\FF {{\mathcal F}}

\def\Z {{\mathbb Z}}

\def\NN {{\mathcal N}}

\def\Q {{\mathbb Q}}

\def\C {{\mathcal C}}

\def\D {{\mathcal D}}

\def\F {{\mathbb F}}
\def\D {{\mathcal D}}
\def\Z {{\mathbb Z}}
\def\Q {{\mathbb Q}}
\def\C {{\mathbb C}}

\def\CC {{\mathcal C}}

\makeatletter
\@namedef{subjclassname@2020}{%
  \textup{2020} Mathematics Subject Classification}
\makeatother

\def\red#1 {\textcolor{red}{#1 }}
\def\blue#1 {\textcolor{blue}{#1 }}

\numberwithin{equation}{section}

\def\Z {{\mathbb Z}}
\begin{document}

\title[Generalized Wall-Sun-Sun primes and monogenic trinomials]{Generalized Wall-Sun-Sun primes and monogenic power-compositional trinomials}

%\author{Joshua Harrington}
%\address{Department of Mathematics, Cedar Crest College, Allentown, Pennsylvania, USA}
%\email[Joshua Harrington]{Joshua.Harrington@cedarcrest.edu}

\author{Lenny Jones}
\address{Professor Emeritus, Department of Mathematics, Shippensburg University, Shippensburg, Pennsylvania 17257, USA}
\email[Lenny~Jones]{doctorlennyjones@gmail.com}

%\author{Daniel White}
%\address{Department of Mathematics, Bryn Mawr College, Bryn Mawr, Pennsylvania 19010-2899, USA}
%\email[Daniel~White]{dfwhite@brynmawr.edu}
\date{\today}

\begin{abstract}
For positive integers $a$ and $b$, we let $[U_n]$  be the Lucas sequence of the first kind defined by
\[U_0=0,\quad U_1=1\quad \mbox{and} \quad U_n=aU_{n-1}+bU_{n-2} \quad \mbox{ for $n\ge 2$},\] and let $\pi(m):=\pi_{(a,b)}(m)$ be the period length of $[U_n]$ modulo the integer $m\ge 2$, where $\gcd(b,m)=1$. We define an \emph{$(a,b)$-Wall-Sun-Sun prime} to be a prime $p$ such that
$\pi(p^2)=\pi(p)$. When $(a,b)=(1,1)$, such a prime $p$ is referred to simply as a \emph{Wall-Sun-Sun prime}.

We say that a monic polynomial $f(x)\in {\mathbb Z}[x]$ of degree $N$ is \emph{monogenic} if $f(x)$ is irreducible over ${\mathbb Q}$ and
\[\{1,\theta,\theta^2,\ldots, \theta^{N-1}\}\]
is a basis for the ring of integers of ${\mathbb Q}(\theta)$, where $f(\theta)=0$.

 Let $f(x)=x^2-ax-b$, and let $s$ be a positive integer. Then, with certain restrictions on $a$, $b$ and $s$,
 we prove that the monogenicity of
 \[f(x^{s^n})=x^{2s^n}-ax^{s^n}-b\] is independent of the positive integer $n$ and is determined solely by whether $s$ has a prime divisor that is an $(a,b)$-Wall-Sun-Sun prime.  This result improves and extends previous work of the author in the special case $b=1$.
\end{abstract}


\maketitle
\section{Introduction}\label{Section:Intro}
Throughout this article, we let $(*)$ denote the set of conditions:
\[(*)\left\{\begin{array}{l}%\label{a-b conditions}
a \mbox{ and } b \mbox{ are positive integers}\\
  a\not \equiv 0 \pmod{4}\\
  b \mbox{ is squarefree}\\
  \D \mbox{ is squarefree,}
\end{array}\right.\] where
 \[\D:=\left\{\begin{array}{cl}
  a^2+4b & \mbox{if $a\equiv 1 \pmod{2}$}\\
   (a/2)^2+b & \mbox{if $a\equiv 0 \pmod{2}$.}
 \end{array} \right.\]

We also let $[U_n]$ be the Lucas sequence of the first kind defined by
\begin{equation}\label{Eq:Lucas}
U_0=0,\quad U_1=1\quad \mbox{and} \quad U_n=aU_{n-1}+bU_{n-2} \quad \mbox{ for $n\ge 2$}.
\end{equation}
The sequence $[U_n]$ is well known to be periodic modulo any integer $m\ge 2$, where $\gcd(b,m)=1$, and we let $\pi(m):=\pi_{(a,b)}(m)$ denote the length of the period of $[U_n]$ modulo $m$.

\begin{defn}\label{Def:(a,b)WSS}
An \emph{$(a,b)$-Wall-Sun-Sun prime} is a prime $p$ with $\gcd(b,p)=1$, such that
\begin{equation}\label{Eq:(a,b)WSS}
\pi(p^2)=\pi(p).
\end{equation}
\end{defn}

We provide some examples of $(a,b)$-Wall-Sun-Sun primes in Table \ref{T:1}.
\begin{table}[h]
 \begin{center}
\begin{tabular}{cc}
 $(a,b)$ & $\{[p,\pi(p^2)]\}$\\ \hline \\[-8pt]
$(2,1)$ & $\{[13,28],[31,30]\}$\\ [2pt]
$(3,26)$ & $\{[71,126]\}$\\ [2pt]
$(10,41)$ & $\{[29,120]\}$\\ [2pt]
$(11,43)$ & $\{[2,3],[5,24]\}$\\ [2pt]
$(15,14)$ & $\{[29,28]\}$\\ [2pt]
$(23,11)$ & $\{[2,3],[3,3],[71,35]\}$\\ [2pt]
$(25,7)$ & $\{[5,8]\}$\\ [2pt]
$(27,22)$ & $\{[13,84]\}$\\ [2pt]
 \end{tabular}
\end{center}
\caption{$(a,b)$-Wall-Sun-Sun primes $p$ and the corresponding period length $\pi(p^2)=\pi(p)$}
 \label{T:1}
\end{table}

 When $(a,b)=(1,1)$, the sequence $[U_n]$ is the well-known Fibonacci sequence, and the $(a,b)$-Wall-Sun-Sun primes in this case are known simply as \emph{Wall-Sun-Sun} primes \cite{CDP,Wiki2}. However, at the time this article was written, no Wall-Sun-Sun primes were known to exist. The existence of Wall-Sun-Sun primes was first investigated by D. D. Wall \cite{Wall} in 1960, and subsequently studied by the Sun brothers \cite{SunSun}, %who showed a connection with Fermat's Last Theorem,
 who showed that the first case of Fermat's Last Theorem is false for exponent $p$ only if $p$ is a Wall-Sun-Sun prime.

 When $b=1$, primes satisfying \eqref{Eq:(a,b)WSS} are also known simply as $a$-Wall-Sun-Sun primes \cite{Wiki1,Wiki2}. We point out that the definition of an $a$-Wall-Sun-Sun prime given in \cite{Wiki1,Wiki2} is a prime $p$ such that
 \begin{equation}\label{Eq:a-WSS}
 U_{\pi(p)}\equiv 0\pmod{p^2}.
 \end{equation}
 In the more general situation of $(a,b)$-Wall-Sun-Sun primes, it is easily seen that condition \eqref{Eq:(a,b)WSS} implies condition \eqref{Eq:a-WSS}. Although it can be shown that the converse is true when $b=1$ \cite{EJ}, % is straightforward to show that the converse is true when $(a,b)=(1,1)$,
 the converse is false in general, as can be seen by the counterexample $(a,b)=(5,2)$ with $p=7$. In this particular example, we have $\pi(7)=48$ and $U_{48}\equiv 0 \pmod{49}$, but $\pi(49)=7\pi(7)=336$. Since Wall was originally concerned with whether there exist any primes $p$ such that \eqref{Eq:(a,b)WSS} holds in the case of $(a,b)=(1,1)$, we have chosen to use condition \eqref{Eq:(a,b)WSS}, instead of condition \eqref{Eq:a-WSS}, for our definition of the more general $(a,b)$-Wall-Sun-Sun prime.

Let $\Delta(f)$ and $\Delta(K)$ denote, respectively, the discriminants over $\Q$ of $f(x)\in \Z[x]$ and a number field $K$. We define $f(x)\in \Z[x]$ to be \emph{monogenic} if $f(x)$ is monic, irreducible over $\Q$ and
  \[\Theta=\{1,\theta,\theta^2,\ldots ,\theta^{\deg(f)-1}\}\] is a basis for the ring of integers $\Z_K$ of $K=\Q(\theta)$, where $f(\theta)=0$. If $\Theta$ fails to be a basis for $\Z_K$, we say that $f(x)$ is \emph{non-monogenic}.
 If $f(x)$ is irreducible over $\Q$ with $f(\theta)=0$,  then \cite{Cohen}
\begin{equation} \label{Eq:Dis-Dis}
\Delta(f)=\left[\Z_K:\Z[\theta]\right]^2\Delta(K).
\end{equation}
Observe then, from \eqref{Eq:Dis-Dis}, that $f(x)$ is monogenic if and only if $\Delta(f)=\Delta(K)$.
Thus, if $\Delta(f)$ is squarefree, then $f(x)$ is monogenic from \eqref{Eq:Dis-Dis}. However, the converse does not hold in general, and when $\Delta(f)$ is not squarefree, it can be quite difficult to determine whether $f(x)$ is monogenic.

In this article, we establish a connection between $(a,b)$-Wall-Sun-Sun primes and the monogenicity of certain power-compositional trinomials. More precisely, we prove
\begin{thm}\label{Thm:Main}
Let $f(x)=x^2-ax-b\in \Z[x]$, where $a$ and $b$ satisfy $(*)$. Let $s\ge 1$ be an integer
such that $\gcd(b,s)=1$, $\delta_p\ne 1$ for each prime divisor $p\ge 3$ of $s$ and $\delta_3=-1$ if $3\mid s$,
 where $\delta_p$ is the Legendre symbol $\left(\frac{\D}{p}\right)$. For any integer $n\ge 1$, define $\FF_n(x):=f(x^{s^n})$. Then $\FF_n(x)$ is monogenic if and only if no prime divisor of $s$ is an $(a,b)$-Wall-Sun-Sun prime. % with $\delta_p=-1$.
 %\begin{enumerate}
  % \item \label{I1:MainThm} If $s$ has a prime divisor $p$, such that $p$ is an $(a,b)$-Wall-Sun-Sun prime with $\delta_p=-1$, then $\FF_n(x)$ is non-monogenic %for all integers $n\ge 1$,
   %\item \label{I2:MainThm} If $s$ has no prime divisor $p$, such that $p$ is an $(a,b)$-Wall-Sun-Sun prime with $\delta_p=-1$, then $\FF_n(x)$ is monogenic %for all integers $n\ge 1$.
 %\end{enumerate}
 \end{thm}

  Theorem \ref{Thm:Main} improves and extends previous work of the author on the special case of $b=1$ \cite{JonesEJM}, which was, in part, originally motivated by recent results of Bouazzaoui \cite{Bouazzaoui1,Bouazzaoui2}. Bouazzaoui showed, under certain conditions on the prime $p\ge 3$,  that
  \[\Q(\sqrt{d}) \quad \mbox{is $p$-rational if and only if} \quad \pi_{(a,b)}(p^2)\ne \pi_{(a,b)}(p),\]
   where $d>0$ is a fundamental discriminant \cite{Wiki0}, $a=\varepsilon+\overline{\varepsilon}$ and $b=-\NN_{\Q(\sqrt{d})/\Q}(\varepsilon)$, with $\varepsilon$ equal to the fundamental unit of $\Q(\sqrt{d})$. We recall, for a prime $p\ge 3$, that a number field $K$ is said to be \emph{$p$-rational} if the Galois group
of the maximal pro-$p$-extension of $K$ which is unramified outside $p$ is a free pro-$p$-group of rank
$r_2 + 1$, where $r_2$ is the number of pairs of complex embeddings of $K$.

\section{Preliminaries}\label{Section:Prelim}
The formula for the discriminant of an arbitrary monic trinomial, due to Swan \cite{Swan}, is given in the following theorem.
\begin{thm}
\label{Thm:Swan}
Let $f(x)=x^N+Ax^M+B\in \Z[x]$, where $0<M<N$. Let $r=\gcd(N,M)$, $N_1=N/r$ and $M_1=M/r$. Then
\begin{equation*}\label{Eq:Del(f)}
\Delta(f)=(-1)^{N(N-1)/2}B^{M-1}D^r,%\left(N^{N/r}B^{(N-M)/r}-(-1)^{N/r}(N-M)^{(N-M)/r}M^{M/r}A^{N/r}\right)^r.
\end{equation*} where
\begin{equation}\label{Eq:D}
D:=N^{N_1}B^{N_1-M_1}-(-1)^{N_1}M^{M_1}(N-M)^{N_1-M_1}A^{N_1}.
\end{equation}
\end{thm}

The next two theorems are due to Capelli \cite{S}.
 \begin{thm}\label{Thm:Capelli1}  Let $f(x)$ and $h(x)$ be polynomials in $\Q[x]$ with $f(x)$ irreducible. Suppose that $f(\alpha)=0$. Then $f(h(x))$ is reducible over $\Q$ if and only if $h(x)-\alpha$ is reducible over $\Q(\alpha)$.
 \end{thm}

\begin{thm}\label{Thm:Capelli2}  Let $c\in \Z$ with $c\geq 2$, and let $\alpha\in\C$ be algebraic.  Then $x^c-\alpha$ is reducible over $\Q(\alpha)$ if and only if either there is a prime $p$ dividing $c$ such that $\alpha=\beta^p$ for some $\beta\in\Q(\alpha)$ or $4\mid c$ and $\alpha=-4\beta^4$ for some $\beta\in\Q(\alpha)$.
\end{thm}

The next proposition follows from Proposition 1 in \cite{Y}.
\begin{prop}\label{Prop:Yokoi}
%Let $k\ge 1$ be an integer, such that $k\ne 4$ and $\D$ is squarefree.
Let $b=1$. Then $\alpha=(a+\sqrt{a^2+4})/2$ is the fundamental unit of $\Q(\sqrt{\D})$ with $\NN(\alpha)=-1$, where $\NN:=\NN_{\Q(\alpha)/\Q}$ denotes the algebraic norm.
\end{prop}

In the sequel, for an integer $m\ge 2$, we let $\ord_m(*)$ denote the order of $*$ modulo $m$, and we define $(a,b)_m:=(a \pmod{m},b \pmod{m})$.
For brevity of notation, we also define
\[\lambda:=\ord_p(b^2) \quad\mbox{and}\quad \delta_p:=\left(\frac{\D}{p}\right),\]
 where $\left(\frac{\D}{p}\right)$ is the Legendre symbol.

 The following theorem is a compilation of results from various sources.
\begin{thm}\label{Thm:Period}
  Let $[U_n]$ be the Lucas sequence as defined in \eqref{Eq:Lucas}.
  Let $p$ be a prime with $b\not \equiv 0 \pmod{p}$.
  \begin{enumerate}
  \item \label{I-1} $\pi(p)=2$ if and only if $(a,b)_p=(0,1)$.
   \item \label{I0:p=2} If $p=2$, then \[\pi(2)=\left\{\begin{array}{cl}
    2 & \mbox{if $(a,b)_4\in \{(2,1),(2,3)\}$}\\
    3 & \mbox{if $(a,b)_4\in \{(1,1),(1,3),(3,1),(3,3)\}$}
  \end{array}\right.\]
  \[  \mbox{and}\quad \pi(4)=\left\{\begin{array}{cl}
    3 & \mbox{if $(a,b)_4=(3,3)$}\\
    4 & \mbox{if $(a,b)_4\in \{(2,1),(2,3)\}$}\\
    6 & \mbox{if $(a,b)_4\in \{(1,1),(1,3),(3,1)\}$.}
  \end{array}\right.\]
   %\item \label{I:even} If $p\ge 3$, then $\pi_k(p)\equiv  0\pmod{2}$.
    \item  \label{I4:R} If $p\ge 3$, then $\pi(p^2)\in \{\pi(p),p\pi(p)\}$.
    %\item  \label{I1:QR} If $p\ge 3$ and $\delta_p=1$, then $p-1\equiv 0 \pmod{\pi(p)}$.
    \item  \label{I2:QNR} If $\delta_p=-1$, then $2(p+1)\lambda\equiv 0 \pmod{\pi(p)}$.
    %\item  \label{I3:R} If $p\ge 3$ and $\D\equiv 0 \pmod{p}$, then $\pi_k(p)=4p$.
     \end{enumerate}
\end{thm}
\begin{proof}
Item \ref{I-1} is obvious. Item \ref{I0:p=2} follows easily by direct calculation, recalling that $a\not \equiv 0 \pmod{4}$ from conditions $(*)$. %item \ref{I:even} is a special case of work found in \cite{FP},
 Item \ref{I4:R} can be found in \cite{Renault}, while item \ref{I2:QNR} follows from a theorem in \cite{GRS}. %items \ref{I1:QR} and \ref{I2:QNR} follow from two theorems in \cite{GRS}.
\end{proof}

The following theorem, known as \emph{Dedekind's Index Criterion}, or simply \emph{Dedekind's Criterion} if the context is clear, is a standard tool used in determining the monogenicity of a polynomial.
\begin{thm}[Dedekind \cite{Cohen}]\label{Thm:Dedekind}
Let $K=\Q(\theta)$ be a number field, $T(x)\in \Z[x]$ the monic minimal polynomial of $\theta$, and $\Z_K$ the ring of integers of $K$. Let $p$ be a prime number and let $\overline{ * }$ denote reduction of $*$ modulo $p$ (in $\Z$, $\Z[x]$ or $\Z[\theta]$). Let
\[\overline{T}(x)=\prod_{i}\overline{\tau_i}(x)^{e_i}\]
be the factorization of $T(x)$ modulo $p$ in $\F_p[x]$, and set
\[g(x)=\prod_{i}\tau_i(x),\]
where the $\tau_i(x)\in \Z[x]$ are arbitrary monic lifts of the $\overline{\tau_i}(x)$. Let $h(x)\in \Z[x]$ be a monic lift of $\overline{T}(x)/\overline{g}(x)$ and set
\[F(x)=\dfrac{g(x)h(x)-T(x)}{p}\in \Z[x].\]
Then
\[\left[\Z_K:\Z[\theta]\right]\not \equiv 0 \pmod{p} \Longleftrightarrow \gcd\left(\overline{F},\overline{g},\overline{h}\right)=1 \mbox{ in } \F_p[x].\]
\end{thm}

The next result is essentially an algorithmic adaptation of Theorem \ref{Thm:Dedekind} specifically for trinomials. %, and will be our main tool for this section.
\begin{thm}{\rm \cite{JKS2}}\label{Thm:JKS}
Let $N\ge 2$ be an integer.
Let $K=\Q(\theta)$ be an algebraic number field with $\theta\in \Z_K$, the ring of integers of $K$, having minimal polynomial $f(x)=x^{N}+Ax^M+B$ over $\Q$, with $\gcd(M,N)=r$, $N_1=N/r$ and $M_1=M/r$. Let $D$ be as defined in \eqref{Eq:D}. A prime factor $p$ of $\Delta(f)$ does not divide $\left[\Z_K:\Z[\theta]\right]$ if and only if $p$ satisfies one of the following items: % all of the following statements are true: %$p$ satisfies one of the following conditions:
\begin{enumerate}[font=\normalfont]
  \item \label{JKS:I1} when $p\mid A$ and $p\mid B$, then $p^2\nmid B$;
  \item \label{JKS:I2} when $p\mid A$ and $p\nmid B$, then
  \[\mbox{either } \quad p\mid A_2 \mbox{ and } p\nmid B_1 \quad \mbox{ or } \quad p\nmid A_2\left((-B)^{M_1}A_2^{N_1}-\left(-B_1\right)^{N_1}\right),\]
  where $A_2=A/p$ and $B_1=\frac{B+(-B)^{p^e}}{p}$ with $p^e\mid\mid N$;
  \item \label{JKS:I3} when $p\nmid A$ and $p\mid B$, then
  \[\mbox{either } \quad p\mid A_1 \mbox{ and } p\nmid B_2 \quad \mbox{ or } \quad p\nmid A_1B_2^{M-1}\left((-A)^{M_1}A_1^{N_1-M_1}-\left(-B_2\right)^{N_1-M_1
  }\right),\]
  where $A_1=\frac{A+(-A)^{p^j}}{p}$ with $p^j\mid\mid (N-M)$, and $B_2=B/p$;
  \item \label{JKS:I4} when $p\nmid AB$ and $p\mid M$ with $N=up^m$, $M=vp^m$, $p\nmid \gcd\left(u,v\right)$, then the polynomials
   \begin{align*}
    G(x):&=x^{N/p^m}+Ax^{M/p^m}+B \quad \mbox{and}\\
    H(x):&=\dfrac{Ax^{M}+B+\left(-Ax^{M/p^m}-B\right)^{p^m}}{p}
   \end{align*}
   are coprime modulo $p$;
         \item \label{JKS:I5} when $p\nmid ABM$, then $p^2\nmid D/r^{N_1}$.
     %\[p^2\nmid \left(B^{N_1-M_1}N_1^{N_1}-(-1)^{M_1}A^{N_1}M_1^{M_1}(M_1-N_1)^{N_1-M_1} \right).\]
   \end{enumerate}
\end{thm}
\begin{rem}
We will find both Theorem \ref{Thm:Dedekind} and Theorem \ref{Thm:JKS} useful in our investigations.
\end{rem}

The next theorem follows from Corollary (2.10) in \cite{Neukirch}.
\begin{thm}\label{Thm:CD}%{\rm \cite{Neukirch}}
  Let $K$ and $L$ be number fields with $K\subset L$. Then \[\Delta(K)^{[L:K]} \bigm|\Delta(L).\]
\end{thm}

\section{The Proof of Theorem \ref{Thm:Main}}
Throughout this section we let
\[f(x)=x^2-ax-b\in \Z[x] \quad \mbox{and} \quad \alpha=\frac{a+\sqrt{a^2+4b}}{2},\] where $a$ and $b$ satisfy $(*)$. %, so that $f(\alpha)=0$.
We first prove some lemmas.
\begin{lemma}\label{Lem:Irreducible}
Let $s$ be a positive integer. Then $f(x^{s^n})$ is irreducible over $\Q$ for all integers $n\ge 1$.
\end{lemma}
\begin{proof}
Since $\D>1$ is squarefree, it follows that $f(x)$ is irreducible over $\Q$, and the trivial case of $s=1$ is true. %The case $b=1$ is proven in \cite{JonesEJM}.
So, suppose then that $s\ge 2$. Note that $f(\alpha)=0$. Let $h(x)=x^{s^n}$ and assume, by way of contradiction, that $f(h(x))$ is reducible. Then, by Theorems \ref{Thm:Capelli1} and \ref{Thm:Capelli2}, we have, for some $\beta\in \Q(\alpha)$, that either $\alpha=\beta^p$ for some prime $p$ dividing $s$, or $\alpha=-4\beta^4$ if $s^n\equiv 0 \pmod{4}$.

If $b\ge 2$, then, in either case, we arrive at a contradiction by taking norms, since $\NN(\alpha)=-b$ is squarefree but neither $\NN(\beta^p)=\NN(\beta)^p$ nor $\NN(-4\beta^4)=16\NN(\beta)^4$ is squarefree. Suppose then that $b=1$. If $\alpha=-4\beta^4$, then
\[-1=\NN(\alpha)=\NN(-4\beta^4)\equiv 0 \pmod{16},\] which is impossible.
Hence, $\alpha=\beta^p$ for some prime divisor $p$ of $s$. %In this case, we have that $\alpha$ is the fundamental unit of $\Q(\sqrt{\D})$ by Proposition \ref{Prop:Yokoi}.
Then, we see by taking norms that
  \[\NN(\beta)^p=\NN(\alpha)=-1,\] which implies that $p\ge 3$ and $\NN(\beta)=-1$, since $\NN(\beta)\in \Z$. Thus, $\beta$ is a unit, and therefore $\beta=\pm \alpha^j$ for some $j\in \Z$, since $\alpha$ is the fundamental unit of $\Q(\sqrt{\D})$ by Proposition \ref{Prop:Yokoi}. %=\Q(\alpha)
   Consequently,
  \[\alpha=\beta^p=(\pm 1)^p\alpha^{jp},\] which implies that $(\pm 1)^p\alpha^{jp-1}=1$, contradicting the fact that $\alpha$ has infinite order in the group of units of the ring of algebraic integers in the real quadratic field $\Q(\sqrt{\D})$.
\end{proof}
\begin{rem}
  Although here we are assuming that conditions $(*)$ hold, so that $a\not \equiv 0 \pmod{4}$, the argument given in the proof of Lemma \ref{Lem:Irreducible} for the case of $b=1$ is still valid when $a\equiv 0 \pmod{4}$ with the single exception of $a=4$ \cite{Y} since, in that case, $\alpha=2+\sqrt{5}$ is not the fundamental unit of $\Q(\sqrt{5})$. However, since $\varepsilon=(1+\sqrt{5})/2$ is the fundamental unit of $\Q(\sqrt{5})$, and $\alpha=\varepsilon^3$, Theorem \ref{Thm:Capelli1} and Theorem \ref{Thm:Capelli2} can be used to determine exactly when $f(x^{s^n})=x^{2s^n}-4x^{s^n}-1$ is reducible and how $f(x^{s^n})$ factors.
\end{rem}
\begin{lemma}\label{Lem:Basic1}
The polynomial $f(x)$ is monogenic.
  \end{lemma}
\begin{proof}
By Lemma \ref{Lem:Irreducible}, $f(x)$ is irreducible over $\Q$. Let $p$ be a prime divisor of $\Delta(f)=a^2+4b$. To examine the monogenicity of $f(x)$, we use Theorem \ref{Thm:JKS} with $\theta=\alpha$.

Suppose first that $p\mid a$. Then $p\mid 4b$. If $p\mid b$, then item \ref{JKS:I1} of Theorem \ref{Thm:JKS} applies, and we see that
$\left[\Z_K:\Z[\alpha]\right]\not \equiv 0 \pmod{p}$ since $b$ is squarefree. Now suppose that $p\nmid b$, so that item \ref{JKS:I2} of Theorem \ref{Thm:JKS} applies. Note that $p=2$ since $p\mid 4b$. Hence, $2\mid a$ and $\D=(a/2)^2+b\equiv 1+b \pmod{4}$ since $a\not \equiv 0 \pmod{4}$. Thus, since $\D$ is squarefree and $2\nmid b$, it follows that $b\equiv 1 \pmod{4}$ and therefore,
\[B_1=(-b+b^2)/2=b(b-1)/2\equiv 0 \pmod{2}.\] Also, $A_2=-a/2\equiv 1 \pmod{2}$, since $a\not \equiv 0 \pmod{4}$. Thus,
\[bA_2^2-(-B_1)^2\equiv 1 \pmod{2},\]
from which we conclude that $[\left[\Z_K:\Z[\alpha]\right]\not \equiv 0 \pmod{2}$.

Next, suppose that $p\nmid a$.  Then $p\nmid 4b$ since $p\mid (a^2+4b)$, and so item \ref{JKS:I5} of Theorem \ref{Thm:JKS} applies. Since $\D$ is squarefree and $p\ne 2$, we deduce that $p^2\nmid (a^2+4b)$ and consequently, $\left[\Z_K:\Z[\alpha]\right]\not \equiv 0 \pmod{p}$, which completes the proof.
 \end{proof}

\begin{lemma}\label{Lem:3}
Let $p$ be a prime, with $b\not \equiv 0 \pmod{p}$. %\gcd(b,p)=1$.
\begin{enumerate}
\item \label{I1: p=2} The prime $p=2$ is an $(a,b)$-Wall-Sun-Sun prime if and only if $(a,b)_4=(3,3)$.
\item \label{I2: p>2} If $p\ge 3$ and $a\equiv 0 \pmod{p}$, then $p$ is an $(a,b)$-Wall-Sun-Sun prime if and only if $\ord_{p^2}(b)=\ord_{p}(b)$ and $a \equiv 0\pmod{p^2}$.
%\item \label{I3:p=3 divides a^2+4b} If $a^2+4b\equiv 0 \pmod{3}$, then $p=3$ is an $(a,b)$-Wall-Sun-Sun prime if and only if
%\[(a,b)_9\in \{(1,8),(2,5),(4,2),(5,2),(7,5),(8,8)\}.\]
\item \label{I4:p>3 divides a^2+4b} If $p\ge 5$ and $\delta_p=0$, then p is not an $(a,b)$-Wall-Sun-Sun prime.
\end{enumerate}
\end{lemma}
\begin{proof}
  We see that item \ref{I1: p=2} follows from item \ref{I0:p=2} of Theorem \ref{Thm:Period}.

  To establish item \ref{I2: p>2}, we let $[U_n]_m$ denote the sequence \eqref{Eq:Lucas} reduced modulo the integer $m\in \{p,p^2\}$. Since $a\equiv 0\pmod{p}$, we can write $a=pk$, for some positive integer $k$.
  Then,
  \[[U_n]_p=[0,1,0,b,0,b^2,0,b^3,0,b^4,0,b^5,\ldots]\] and
  \[[U_n]_{p^2}=[0,1,pk,b,2pkb,b^2,3pkb^2,b^3,\ldots,\ord_p(b)pkb^{\ord_p(b)-1},b^{\ord_p(b)},\ldots].\]

   Thus, it follows that $p$ is an $(a,b)$-Wall-Sun-Sun prime if and only if
  \begin{align*}
  \pi(p^2)=\pi(p)=2\ord_p(b)&\Longleftrightarrow \ord_{p^2}(b)=\ord_p(b) \quad \mbox{and}\\
  &\qquad \qquad \ord_p(b)pkb^{\ord_p(b)-1}\equiv 0 \pmod{p^2}\\
  &\Longleftrightarrow \ord_{p^2}(b)=\ord_p(b)\quad \mbox{and}\quad a\equiv 0 \pmod{p^2},\\
  \end{align*}
  since $b\not \equiv 0 \pmod{p}$ and $\ord_p(b)\le p-1\not \equiv 0\pmod{p}$.

  The proof of %items \ref{I3:p=3 divides a^2+4b} and
   item \ref{I4:p>3 divides a^2+4b} can be found in \cite{HJBAMS}.
 \end{proof}

\begin{lemma}\label{Lem:Order}
%Let $f(x)=x^2-ax-b$ and suppose that $f(\alpha)=0$.
Let $\overline{\alpha}=(a-\sqrt{a^2+4b})/2$, and let $p\ge 3$ be a prime such that $\delta_p=-1$. %$\left(\frac{b}{p}\right)=-1$ and
     Then
     \begin{enumerate}
     \item \label{Per I:1} $\ord_m(\alpha)=\ord_m(\overline{\alpha})=\pi(m)$ for $m\in \{p,p^2\}$ and %and $\ord_{p^2}(\varepsilon)=\pi_k(p^2)$
     \item \label{Per I:2} $\alpha^{p+1}\equiv -b\pmod{p}$.
     \end{enumerate}
\end{lemma}
\begin{proof}
Note that $b\not \equiv 0 \pmod{p}$ since $\delta_p=-1$.
  It follows from \cite{Robinson} that the order, modulo an %odd
   integer $m\ge 3$ with $\gcd(m,b)=1$, of the companion matrix $\CC$ for the characteristic polynomial of $[U_n]$ is $\pi(m)$. The characteristic polynomial of $[U_n]$ is $f(x)$, so that
  \[\CC=\left[\begin{array}{cc}
    0&b\\
    1&a
  \end{array}\right].\] Since the eigenvalues of $\CC$ are $\alpha$ and $\overline{\alpha}$, we conclude that
  \[\ord_m\left(\left[\begin{array}{cc}
    \alpha&0\\
    0&\overline{\alpha}
  \end{array}\right]\right)=\ord_m(\CC)=\pi(m), \quad \mbox{for $m\in \{p,p^2\}$.}
  \] Let $z\ge 1$ be an integer, and suppose that $\alpha^z=c+d\sqrt{\D}\in \Q(\sqrt{\D})$. Then
    $\NN(\alpha^z)=c^2-\D d^2$. But $\NN(\alpha^z)=\NN(\alpha)^z=(-b)^z$, so that $c^2-\D d^2=(-b)^z$. Thus,
  \[\overline{\alpha}^z=\left(-b/\alpha\right)^z=(-b)^z/(c+d\sqrt{\D})=(-b)^z(c-d\sqrt{\D})/(c^2-\D d^2)=c-d\sqrt{\D}.\]
  Hence, since $\delta_p=-1$, it follows that
  \[\alpha^z\equiv 1 \pmod{m} \quad \mbox{if and only if} \quad \overline{\alpha}^z\equiv 1 \pmod{m}\]
  for $m\in \{p,p^2\}$,
  which establishes item \ref{Per I:1}.

  By Euler's criterion,
   \[\left(\sqrt{a^2+4b}\right)^{p+1}=(a^2+4b)^{(p-1)/2}(a^2+4b)\equiv \delta_p(a^2+4b)\equiv -(a^2+4b) \pmod{p},\] which implies %that
   \[\left(\sqrt{a^2+4b}\right)^{p}\equiv -\sqrt{a^2+4b} \pmod{p}.\]
   Hence,
   \begin{align*}\label{Eq:Expansion}
     \alpha^{p+1}&=\left(\frac{a+\sqrt{a^2+4b}}{2}\right) \left(\frac{a+\sqrt{a^2+4b}}{2}\right)^{p}\\
     &=\left(\frac{a+\sqrt{a^2+4b}}{2}\right) \sum_{j=0}^p\binom{p}{j}\left(\frac{a}{2}\right)^j\left(\frac{\sqrt{a^2+4b}}{2}\right)^{p-j}\\
     &\equiv \left(\frac{a+\sqrt{a^2+4b}}{2}\right)\left(\left(\frac{a}{2}\right)^p+\left(\frac{\sqrt{a^2+4b}}{2}\right)^{p}\right) \pmod{p}\\
     &\equiv \left(\frac{a+\sqrt{a^2+4b}}{2}\right)\left(\frac{a-\sqrt{a^2+4b}}{2}\right) \pmod{p}\\
     &\equiv -b \pmod{p},
    \end{align*}
       which completes the proof of the lemma.
 \end{proof}

\begin{lemma}\label{Lem:Main1}
%Let $f(x)=x^2-ax-b$ and suppose that $f(\alpha)=0$.
  Let $p\ge 3$ be a prime such that $\delta_p=-1$. Then the following conditions are equivalent:
  \begin{enumerate}
    \item \label{Lem:Main1 I1} $p$ is an $(a,b)$-Wall-Sun-Sun prime,
    \item \label{Lem:Main1 I2} $f(\alpha^{p^m}) \equiv 0\pmod{p^2}$ for all integers $m\ge 1$,
    \item \label{Lem:Main1 I3} $f(\alpha^{p^m}) \equiv 0\pmod{p^2}$ for some integer $m\ge 1$.
  \end{enumerate}
  \end{lemma}
\begin{proof}
First, observe that \ref{Lem:Main1 I2} clearly implies \ref{Lem:Main1 I3}. %For brevity of notation, we let $\lambda=\ord_p(b^2)$ % and $\gamma=\ord_p(b)$
%in the remainder of the proof.

  We show next that \ref{Lem:Main1 I1} implies \ref{Lem:Main1 I2}. Because $p$ is an $(a,b)$-Wall-Sun-Sun prime, we define
  \[\pi:=\pi(p^2)=\pi(p).\]
       Since $\delta_p=-1$, we see from item \ref{I2:QNR} of Theorem \ref{Thm:Period} that
       \[2(p+1)\lambda\equiv 0 \pmod{\pi}.\] %$2(p+1)\lambda\equiv 0 \pmod{\pi(p)}$. Hence, since $p$ is an $(a,b)$-Wall-Sun-Sun prime, it follows that $2(p+1)\lambda\equiv 0 \pmod{\pi(p^2)}$.
        The squares modulo $p$ form a subgroup, of order $(p-1)/2$,  of the multiplicative group $(\Z/p\Z)^{*}$. Thus, $(p-1)/2\equiv 0 \pmod{\lambda}$, so that
       \[2(p+1)(p-1)/2=p^2-1\equiv 0 \pmod{\pi}.\] Consequently, $\alpha^{p^2-1}\equiv 1 \pmod{p^2}$ by item \ref{Per I:1} of Lemma \ref{Lem:Order},
             from which it follows that
       \[\alpha^{p^{2k}}\equiv \alpha \pmod{p^2}\quad \mbox{and}\quad \alpha^{p^{2k+1}}\equiv \alpha^p \pmod{p^2},\]
       for every integer $k\ge 1$.
              Hence,
       \begin{equation*}\label{Eq:f of powers}
       f(\alpha^{p^m})\equiv \left\{\begin{array}{cl}
   \alpha^2-a\alpha-b \pmod{p^2}& \mbox{if $m\equiv 0 \pmod{2}$,}\\
   \alpha^{2p}-a\alpha^{p}-b \pmod{p^2}& \mbox{if $m\equiv 1 \pmod{2}$.}
   \end{array}\right.
   \end{equation*}
        Thus, $f(\alpha^{p^m}) \equiv 0\pmod{p^2}$ when $m\equiv 0 \pmod{2}$, since $\alpha^2-a\alpha-b=0$.
   Suppose then that $m\equiv 1 \pmod{2}$. Let $\overline{\alpha}=(a-\sqrt{a^2+4b})/2$. Since $p$ is an $(a,b)$-Wall-Sun-Sun prime, and the fact that $\overline{\alpha}=-b/\alpha$, we deduce from the Binet-representation formula for $U_{\pi}$ that
       \begin{equation}\label{Eq:Upi}
       U_{\pi}=\frac{\alpha^{\pi}-\overline{\alpha}^{\pi}}{\alpha-\overline{\alpha}}
    =\frac{\alpha^{2\pi}-(-b)^{\pi}}{\alpha^{\pi}\left(\alpha-\overline{\alpha}\right)}\equiv 0 \pmod{p^2}.
     \end{equation} Hence, since $\alpha^{\pi}\equiv 1 \pmod{p^2}$ from item \ref{Per I:1} of Lemma \ref{Lem:Order}, we conclude from \eqref{Eq:Upi} that
    $(-b)^{\pi}\equiv 1 \pmod{p^2}$, which implies that
    \begin{equation}\label{Eq:b2(p+1)lam}
    b^{2(p+1)\lambda}\equiv 1\pmod{p^2},
    \end{equation}
    by item \ref{I2:QNR} of Theorem \ref{Thm:Period}. Thus, from \eqref{Eq:b2(p+1)lam}, it follows that
    \begin{equation}\label{Eq:b2(p+1)lam 2} b^{2(p+1)\lambda}-1\equiv (b^{2\lambda}-1)B\equiv 0 \pmod{p^2},
    \end{equation}
    where \[B=(b^{2\lambda})^p+(b^{2\lambda})^{p-1}+\cdots +b^{2\lambda}+1.\] Since $b^{2\lambda}\equiv (b^2)^{\lambda}\equiv 1 \pmod{p}$, we see that $B\equiv p+1\equiv 1 \pmod{p}$. Therefore,
    \begin{equation}\label{Eq:b2lam}
      b^{2\lambda}-1\equiv 0 \pmod{p^2},
      \end{equation}
    from \eqref{Eq:b2(p+1)lam 2}. Also, since $\delta_p=-1$ and $\alpha^{\pi}\equiv 1 \pmod{p^2}$, we have from item \ref{I2:QNR} of Theorem \ref{Thm:Period} that
    \begin{equation}\label{Eq:alpha2(p+1)lam}
    \alpha^{2(p+1)\lambda}-1\equiv 0 \pmod{p^2}.
    \end{equation}
    Combining \eqref{Eq:b2lam} and \eqref{Eq:alpha2(p+1)lam} yields
    \begin{equation}\label{Eq:Combine}
    \alpha^{2(p+1)\lambda}-b^{2\lambda}\equiv \left(\alpha^{p+1}-b\right)\left(\alpha^{p+1}+b\right)C \equiv 0 \pmod{p^2},
    \end{equation}
    where
    \begin{align*}C&=\left(\alpha^{2(p+1)}\right)^{\lambda-1}+\left(\alpha^{2(p+1)}\right)^{\lambda-2}b^2+\cdots +\alpha^{2(p+1)}(b^2)^{\lambda-2}+(b^2)^{\lambda-1}\\
    &\equiv \lambda b^{2\lambda+2} \pmod{p},
    \end{align*} since $\alpha^{2(p+1)}\equiv b^2 \pmod{p}$ from item \ref{Per I:2} of Lemma \ref{Lem:Order}. Thus, from \eqref{Eq:b2lam} and the fact that $(p-1)/2 \equiv 0 \pmod{\lambda}$, we deduce that
    $C\equiv \lambda b^{2}\not \equiv 0 \pmod{p}$.  Note that $\alpha^{p+1}-b\not \equiv 0\pmod{p}$ since $\alpha^{p+1}+b\equiv 0\pmod{p}$ and $b\not\equiv 0 \pmod{p}$.
    Therefore, it follows from \eqref{Eq:Combine} that $\alpha^{p+1}\equiv -b \pmod{p^2}$. Hence, $\alpha^p\equiv -b\alpha^{-1}\pmod{p^2}$, and consequently,
    \begin{align*}
    f(\alpha^{p^m})&\equiv \alpha^{2p}-a\alpha^p-b \pmod{p^2}\\
    &\equiv \left(-b\alpha^{-1}\right)^2-a\left(-b\alpha^{-1}\right)-b\pmod{p^2}\\
    &\equiv -b\alpha^{-2}(\alpha^2-a\alpha-b)\pmod{p^2}\\
    &\equiv 0 \pmod{p^2}
    \end{align*} since $\alpha^2-a\alpha-b=0$, which completes the proof that \ref{Lem:Main1 I1} implies \ref{Lem:Main1 I2}.

Finally, to establish that \ref{Lem:Main1 I3} implies \ref{Lem:Main1 I1}, we first note that $\pi(p^2)\in \{\pi(p),p\pi(p)\}$
by item \ref{I4:R} of Theorem \ref{Thm:Period}. Then, in either case, we have that $\alpha^{p\pi(p)}\equiv 1 \pmod{p^2}$, and we conclude from item \ref{I2:QNR} of Theorem \ref{Thm:Period} that
$\alpha^{2p(p+1)\lambda}\equiv 1\pmod{p^2}$. Since $(p-1)/2 \equiv 0 \pmod{\lambda}$, we deduce
\[\alpha^{2p(p+1)(p-1)/2}\equiv \alpha^{p^3-p}\equiv 1 \pmod{p^2},\]
so that $\alpha^{p^3}\equiv \alpha^p \pmod{p^2}$.
It then follows easily that %from \eqref{Eq:alphapower mod p^2} that
\begin{equation}\label{Eq:alphapower mod p^2}
\alpha^{p^{2k}}\equiv \alpha^{p^2} \pmod{p^2} \quad \mbox{and}\quad \alpha^{p^{2k+1}}\equiv \alpha^{p} \pmod{p^2},
\end{equation}
for all integers $k\ge 1$. Hence, from \eqref{Eq:alphapower mod p^2}, we have that
  \begin{equation}\label{Eq:f of powers again}
       f(\alpha^{p^m})\equiv \left\{\begin{array}{cl}
   \alpha^{2p^2}-a\alpha^{p^2}-b \pmod{p^2}& \mbox{if $m\equiv 0 \pmod{2}$,}\\
   \alpha^{2p}-a\alpha^{p}-b \pmod{p^2}& \mbox{if $m\equiv 1 \pmod{2}$.}
   \end{array}\right.
   \end{equation}
   
%By Hensel, the only zeros of $f(x)$ in $(\Z/p^2\Z)[\sqrt{\D}]$ are $\alpha$ and $\overline{\alpha}=-b\alpha^{-1}$.
Since $\delta_p=-1$, we have that $f(x)$ is irreducible modulo $p$. Consequently, the only zeros of $f(x)$ in $(\Z/p^2\Z)[\sqrt{5}]$ are $\alpha$ and $\overline{\alpha}=-b\alpha^{-1}$.
Suppose that $f(\alpha^{p^m})\equiv 0\pmod{p^2}$ for some integer $m\equiv 1 \pmod{2}$. Then, we see from \eqref{Eq:f of powers again} that
\[\mbox{either} \quad \alpha^p\equiv \alpha \pmod{p^2} \quad \mbox{or}\quad \alpha^p\equiv \overline{\alpha}\pmod{p^2}.\]
If $\alpha^p\equiv \alpha \pmod{p^2}$, then, from item \ref{Per I:2} of Lemma \ref{Lem:Order}, we have that
\[\frac{a^2+4b+a\sqrt{a^2+4b}}{2}= \alpha^2+b\equiv \alpha^{p+1}+b\equiv 0 \pmod{p},\]
which implies that $a^2+4b\equiv 0 \pmod{p}$, contradicting the fact that $\delta_p=-1$. Hence, %$\alpha^p\equiv \overline{\alpha}\equiv -b\alpha^{-1}\pmod{p^2}$, or equivalently,
\begin{equation}\label{Eq:Condition 1}
  \alpha^p\equiv \overline{\alpha}\equiv -b\alpha^{-1}\pmod{p^2} \quad \mbox{or equivalently,} \quad \alpha^{p+1}\equiv -b\pmod{p^2}.
\end{equation}
Since $\alpha^{2p}-a\alpha^{p}-b \equiv 0 \pmod{p^2}$, then $\overline{\alpha}^{2p}-a\overline{\alpha}^p-b\equiv 0 \pmod{p^2}$ so that
\[\mbox{either} \quad \overline{\alpha}^p\equiv \alpha \pmod{p^2} \quad \mbox{or}\quad \overline{\alpha}^p\equiv \overline{\alpha}\pmod{p^2}.\]
  From \eqref{Eq:Condition 1}, we have that
\begin{equation}\label{Eq:from condition1}
\overline{\alpha}^p\equiv (-b)^p(\alpha^p)^{-1}\equiv (-b)^p(\overline{\alpha})^{-1} \pmod{p^2}.
\end{equation}
If $\overline{\alpha}^p\equiv \overline{\alpha}\pmod{p^2}$, then we see from \eqref{Eq:from condition1} that
% \[\overline{\alpha}\equiv \overline{\alpha}^p\equiv (-b)^p(\alpha^p)^{-1}\equiv (-b)^p(\overline{\alpha})^{-1} \pmod{p^2}.\] Hence,
$\overline{\alpha}^2\equiv (-b)^p \pmod{p^2}$, so that $\overline{\alpha}^2+b\equiv 0 \pmod{p}$. That is, %$(-b)^p\equiv \alpha^{2p} \pmod{p^2}$ from \eqref{Eq:Condition 1}. Hence,  %from the proof of item \ref{Per I:2} of Lemma \ref{Lem:Order}, we have that
\[\left(\frac{a-\sqrt{a^2+4b}}{2}\right)^2+b= \frac{a^2+4b-a\sqrt{a^2+4b}}{2} \equiv 0 \pmod{p},\]
%since $\alpha^p\equiv \overline{\alpha} \pmod{p}$,
which implies that $a^2+4b\equiv 0 \pmod{p}$, again contradicting the fact that $\delta_p=-1$. Therefore, $\overline{\alpha}^p\equiv \alpha \pmod{p^2}$, and we see from \eqref{Eq:Condition 1} and \eqref{Eq:from condition1} that
\[(-b)^p\equiv \alpha^{p+1}\equiv -b \pmod{p^2}.\]
%from \eqref{Eq:Condition 1}.
Consequently,
\begin{equation}\label{Eq:Condition 2}
(-b)^{p-1}\equiv 1 \pmod{p^2}.
\end{equation}
Combining \eqref{Eq:Condition 1} and \eqref{Eq:Condition 2} yields
\begin{equation}\label{Eq:alphaorder}
\alpha^{p^2-1}\equiv (\alpha^{p+1})^{p-1}\equiv (-b)^{p-1}\equiv 1 \pmod{p^2}.
\end{equation}
Recall that $\pi(p^2)\in \{\pi(p),p\pi(p)\}$ by item \ref{I4:R} of Theorem \ref{Thm:Period}.
If $\pi(p^2)=p\pi(p)$, then we have by item \ref{Per I:1} of Lemma \ref{Lem:Order} and \eqref{Eq:alphaorder} that $p^2-1\equiv 0 \pmod{p}$, which is impossible. Thus, we must have $\pi(p^2)=\pi(p)$, which completes the proof that \ref{Lem:Main1 I3} implies \ref{Lem:Main1 I1} when $m\equiv 1 \pmod{2}$.
%Since similar arguments establish this implication when $m\equiv 0 \pmod{2}$, we omit the details.\\
%\begin{comment}
  %\red{Details for $m\equiv 0 \pmod{2}$.}\\

  Suppose now that $f(\alpha^{p^m})\equiv 0\pmod{p^2}$ for some integer $m\equiv 0 \pmod{2}$. Then, we see from \eqref{Eq:f of powers again} that
\begin{equation}\label{Eq:m=0 choices}
\mbox{either} \quad \alpha^{p^2}\equiv \alpha \pmod{p^2} \quad \mbox{or}\quad \alpha^{p^2}\equiv \overline{\alpha}\equiv -b\alpha^{-1}\pmod{p^2}.
\end{equation}
If $\alpha^{p^2}\equiv  -b\alpha^{-1}\pmod{p^2}$, then $\alpha^{p^2+1}\equiv -b \pmod{p^2}$. Note that $p\nmid b$ since $\delta_p=-1$. Hence, using item \ref{Per I:2} of Lemma \ref{Lem:Order}, we have that
\begin{align*}\label{Eq:m=0}
  \alpha^{p^2+1}\equiv -b \pmod{p}&\Longrightarrow \left(-b\alpha^{-1}\right)^p\alpha\equiv-b \pmod{p}\\
  &\Longrightarrow (-b)^p(\alpha^p)^{-1}\alpha\equiv -b \pmod{p}\\
  &\Longrightarrow \alpha\equiv \alpha^p \pmod{p}\\ %\mbox{(since $p\nmid b$)}
  &\Longrightarrow \alpha\equiv \overline{\alpha}, \pmod{p}
\end{align*} which yields the contradiction $a^2+4b\equiv 0 \pmod{p}$. Thus, $\alpha^{p^2}\equiv \alpha \pmod{p^2}$ from \eqref{Eq:m=0 choices}, which implies that $\alpha^{p^2-1}\equiv 1 \pmod{p^2}$. Therefore, the conclusion of the proof in this case is identical to the conclusion of the case $m\equiv 1 \pmod{2}$ following \eqref{Eq:alphaorder}.
%\end{comment}
\end{proof}

\begin{lemma}\label{Lem:Main2}
%Let $f(x)=x^2-ax-b$ and suppose that $f(\alpha)=0$.
Let $s\ge 3$ and $n\ge 1$ be integers. Let $p\ge 3$ be a prime such that %$p\nmid a$,
%$p\nmid b$ and
$p^m\mid \mid s^n$ with $m\ge 1$. Suppose that $\FF_n(x):=f(x^{s^n})$ and $K=\Q(\theta)$, with $\Z_K$ the ring of integers of $K$, where $\FF_n(\theta)=0$.
\begin{enumerate}
\item \label{I1:LemMain2} If  $\delta_p\ne 1$, then
\[[\Z_{K}:\Z[\theta]]\equiv 0\pmod{p}\quad \mbox{if and only if}\quad \alpha^{2p^{m}}-a\alpha^{p^{m}}-b\equiv 0 \pmod{p^2}.\]
%\[-\overline{\left(\frac{\alpha^{2p^{m}}-a\alpha^{p^{m}}-b}{p}\right)}\ne 0.\]
\item \label{I2:LemMain2} Furthermore, if  $\delta_p=0$, then $[\Z_{K}:\Z[\theta]]\not \equiv 0\pmod{p}$.
\end{enumerate}
\end{lemma}
\begin{proof}
For item \ref{I1:LemMain2}, we apply Theorem \ref{Thm:Dedekind} to $T(x):=\FF_n(x)$ using the prime $p$.
 Let
 \begin{equation}\label{Eq:tau}
 \tau(x)=x^{2s^n/p^{m}}-ax^{s^n/p^{m}}-b=f\left(x^{s^n/p^{m}}\right),
 \end{equation} and
$\overline{\tau}(x)=\prod_{i}\overline{\tau_i}(x)^{e_i}$, %x^{2s^n/p^{m}}-\overline{a}x^{s^n/p^{m}}-\overline{b}=
where the $\overline{\tau_i}(x)$ are irreducible in $\F_p[x]$.
Then $\overline{T}(x)=\prod_{i}\overline{\tau_i}(x)^{p^{m}e_i}$. Thus, we can let
\[g(x)=\prod_{i}\tau_i(x) \quad \mbox{and}\quad h(x)=\prod_{i}\tau_i(x)^{p^{m}e_i-1},\] where the $\tau_i(x)$ are monic lifts of the $\overline{\tau_i}(x)$. Note also that
\[g(x)h(x)=\prod_{i}\tau_i(x)^{p^{m}e_i}=\tau(x)+pw(x),\]
for some $w(x)\in \Z[x]$. Then, in Theorem \ref{Thm:Dedekind}, we have that
   \begin{align}\label{Eq:pF}
   \begin{split}
   pF(x)&=g(x)h(x)-T(x)\\
   &=(\tau(x)+pw(x))^{p^{m}}-T(x)\\
   &=\sum_{j=1}^{p^{m}-1}\binom{p^{m}}{j}\tau(x)^j(pw(x))^{p^{m}-j}+\left(pw(x)\right)^{p^{m}}+\tau(x)^{p^{m}}-T(x)\\
   &\equiv \tau(x)^{p^{m}}-T(x) \pmod{p^2}.
   \end{split}
   \end{align}
   Suppose that $\tau(\gamma)=0$. Then, we see from \eqref{Eq:tau} that
   \[\tau(\gamma)=f(\gamma^{s^n/p^{m}})=0.\]
  If $\delta_p=-1$, then $f(x)$ is irreducible modulo $p$, while if $\delta_p=0$, then $f(x)\equiv (x-\alpha)^2 \pmod{p}$, where $\alpha=2^{-1}a \pmod{p}$. In either case, without loss of generality,
 we can assume $\gamma^{s^n/p^{m}}=\alpha$ so that $\gamma^{s^n}=\alpha^{p^{m}}$.
   Hence, from \eqref{Eq:pF}, it follows that
   \begin{align*}
     pF(\gamma)&\equiv -T(\gamma) \pmod{p^2}\\
     &\equiv -\left(\gamma^{2s^n}-a\gamma^{s^n}-b\right) \pmod{p^2}\\
     &\equiv -\left(\alpha^{2p^{m}}-a\alpha^{p^{m}}-b\right) \pmod{p^2},
   \end{align*}
   which completes the proof of item \ref{I1:LemMain2}.
  % \[\overline{F}(x)=\overline{\left(\dfrac{\tau(x)^{p^{m}}-T(x)}{p}\right)}.\]

   To establish item \ref{I2:LemMain2}, we show that
  \[S(x):=x^{2p^m}-ax^{p^m}-b\] has no zeros modulo $p^2$ for any integer $m\ge 1$. Observe that
  \[S(x)\equiv f(x)^{p^m}\equiv (x-\alpha)^{2p^m} \pmod{p},\] %\equiv (x-2^{-1}a)^{2p^m}
  where $\alpha=2^{-1}a \pmod{p}$. Since $S^{\prime}(\alpha)\equiv 0 \pmod{p}$, we deduce from Hensel that either $S(x)$ has no zeros modulo $p^2$, or has the $p$ zeros:
  \[\alpha,\quad \alpha+p,\quad  \alpha+2p,\quad  \ldots ,\quad \alpha+(p-1)p \quad \mbox{modulo $p^2$}.\]
 % where $\alpha=2^{-1}a \pmod{p}$.
    Suppose then, by way of contradiction, that
    \begin{equation}\label{Eq:S}
    S(\alpha)=\alpha^{2p^m}-a\alpha^{p^m}-b\equiv 0 \pmod{p^2}.
    \end{equation} Since $\delta_p=0$, then $\alpha^2\equiv 2^{-2}a^2\equiv  -b \pmod{p}$, and therefore,
    \begin{equation}\label{Eq:p^2}
    \alpha^{2p^m}\equiv (-b)^{p^m}\pmod{p^2}.
    \end{equation}
    Hence, from \eqref{Eq:S} and \eqref{Eq:p^2}, we have that
    \begin{equation}\label{Eq:Start}
    a\alpha^{p^m}\equiv (-b)^{p^m}-b \pmod{p^2}.
    \end{equation} Then, squaring both sides of \eqref{Eq:Start} and using \eqref{Eq:p^2} again gives
    \begin{equation}\label{Eq:Next}
      a^2(-b)^{p^m}\equiv \left((-b)^{p^m}-b\right)^2 \pmod{p^2},
    \end{equation}
    which in turn yields
    \begin{align}\label{Eq:Final}
     %\Longrightarrow &a^2(-b)^{p^m}\equiv a^2\alpha^{2p^m}\equiv \left((-b)^{p^m}-b\right)^2 \pmod{p^2} \nonumber \\
      \Longrightarrow & a^2(-b)^{p^m-2}\equiv \left((-b)^{p^m-1}+1\right)^2 \pmod{p^2} \nonumber \\
      \Longrightarrow & a^2(-b)^{p^m-2}-4(-b)^{p^m-1}\equiv \left((-b)^{p^m-1}+1\right)^2-4(-b)^{p^m-1} \pmod{p^2} \nonumber \\
      \Longrightarrow &  \left(a^2+4b\right)(-b)^{p^m-2}\equiv \left((-b)^{p^m-1}-1\right)^2\equiv 0 \pmod{p^2},
    \end{align}
    since $(-b)^{p^m-1}-1\equiv 0 \pmod{p}$. Thus, since $p\nmid b$, we conclude from \eqref{Eq:Final} that
    \[a^2+4b\equiv 0 \pmod{p^2},\]
    which contradicts the fact that $\D$ is squarefree, and completes the proof of the lemma.
 \end{proof}
\begin{rem}
In the context of Lemma \ref{Lem:Main2}, item \ref{I2:LemMain2} shows that a prime $p$ with $\delta_p=0$ cannot ``cause" $\FF_n(x)$ to be non-monogenic.
\end{rem}

 \begin{proof}[Proof of Theorem \ref{Thm:Main}]
 %For brevity of notation, define
 %\[\FF_n(x):=f(x^{s^n})=x^{2s^n}-ax^{s^n}-b \quad \mbox{for $n\ge 0$.}\]
  Note that $\FF_0(x)=f(x)$. We have that $\FF_n(x)$ is irreducible for all $n\ge 0$ by Lemma \ref{Lem:Irreducible}.
  By Theorem \ref{Thm:Swan},
 \begin{equation}\label{Eq:Disc}
 \Delta(\FF_n)=(-b)^{s^n-1}s^{2ns^n}(a^2+4b)^{s^n}.
 \end{equation}

  ($\Rightarrow$) We prove the contrapositive. Assume that $s$ has a prime divisor $p$ that is an $(a,b)$-Wall-Sun-Sun prime, and that $p^m\mid \mid s^n$, with $m\ge 1$. % with $\delta_p=-1$. Note that $p\nmid b$ since $\delta_p=-1$.

  Suppose first that $p=2$, and write $s^n=2^mv$, with $2\nmid v$. Then $2\nmid a$ since $(a,b)=(3,3)_4$ by item \ref{I1: p=2} of Lemma \ref{Lem:3}. Applying item \ref{JKS:I4} of Theorem \ref{Thm:JKS} to $\FF_n(x)$ we see that
  \begin{align}\label{Eq:H}
  G(x)&=x^{2s^n/2^m}-ax^{s^n/2^m}-b \nonumber\\
  &=x^{2v}-ax^v-b \nonumber\\
  &\equiv x^{2v}+x^v+1 \pmod{2}\quad \mbox{and} \nonumber\\
  H(x)&=\frac{-ax^{2^mv}-b+(ax^v+b)^{2^m}}{2} \nonumber\\
  &=\left(\frac{a^{2^m}-a}{2}\right)x^{2^mv}+\sum_{j=1}^{2^m-1}\frac{\binom{2^m}{j}}{2}(ax^v)^jb^{2^m-j}+\frac{b^{2^m}-b}{2}\\
  &\equiv \left(x^{2v}+x^v+1\right)^{2^{m-1}} \pmod{2}, \nonumber
  \end{align}
  since $a^{2^m}-a\equiv b^{2^m}-b\equiv 2 \pmod{4}$ and \cite{HK} %$\frac{a^{2^m}-a}{2} \equiv \frac{b^{2^m}-b}{2} \equiv 1 \pmod{2}$ and \cite{HK}
  \begin{equation}\label{Eq:Binom}
  \binom{2^m}{j}\equiv \left\{\begin{array}{cl}
  0 \pmod{4} & \mbox{if $j\ne 2^{m-1}$}\\
  2 \pmod{4} & \mbox{if $j=2^{m-1}$.}
  \end{array}\right.
  \end{equation}
  Thus, $G(x)$ and $H(x)$ are not coprime in $\F_2[x]$, and therefore, $\FF_n(x)$ is not monogenic.

  Suppose next that $p\ge 3$. Recall that if $p=3$, then $\delta_3=-1$ by hypothesis. If $p\ge 5$, then, since $p$ is an $(a,b)$-Wall-Sun-Sun prime and $\delta_p\ne 1$, we conclude that $\delta_p=-1$, from  item \ref{I4:p>3 divides a^2+4b} of Lemma \ref{Lem:3}. Thus, $\FF_n(x)$ is non-monogenic by Lemma \ref{Lem:Main1} and item \ref{I1:LemMain2} of Lemma \ref{Lem:Main2}, which completes the proof in this direction.

  \begin{comment} it is enough to show that $\FF_1(x)=x^{2s}-ax^s-b$  is not monogenic. Let $\FF_1(\theta)=0$,  $K=\Q(\theta)$ and $\Z_K$ be the ring of integers of $K$. Let $p^m\mid\mid s$ with $m\ge 1$.

  If $p\mid a$, then $p^2\mid a$ and $\ord_{p^2}(b)=\ord_{p}(b)$  by item \ref{I2: p>2} of Lemma \ref{Lem:3}, since $p$ is an $(a,b)$-Wall-Sun-Sun prime. Since $p\nmid b$, we can apply item \ref{JKS:I2} of Theorem \ref{Thm:JKS} to $\FF_1(x)$ with $e=m$. Then
 \begin{align*}
 A_2&=-\frac{a}{p}\equiv 0 \pmod{p} \quad \mbox{and}\\
 B_1&=\frac{b^{p^m}-b}{p}=\frac{b\left(b^{p^m-1}-1\right)}{p}\equiv 0 \pmod{p},
 \end{align*}
 since $\ord_{p^2}(b)=\ord_{p}(b)$. Thus, $\left[\Z_K:\Z[\theta]\right]\equiv 0\pmod{p}$ and $\FF_1(x)$ is not monogenic.

 Assume next that $p\nmid a$. Since $\delta_p=-1$,  we have that
\[\alpha^{2p^m}-a\alpha^{p^m}-b \equiv 0\pmod{p^2}\] by Lemma \ref{Lem:Main1}. Hence, it follows from Lemma \ref{Lem:Main2} that $\left[\Z_K:\Z[\theta]\right]\equiv 0\pmod{p}$, and therefore, $\FF_1(x)$ is not monogenic, which completes the proof in this direction.
\end{comment}

($\Leftarrow$) Note that when $s=1$, we have that $\FF_n(x)=f(x)$ for all $n\ge 0$, and so $\FF_n(x)$ is monogenic by Lemma \ref{Lem:Basic1}. So, assume that $s\ge 2$, and suppose that no prime divisor of $s$ is an $(a,b)$-Wall-Sun-Sun prime.

    %By Lemma \ref{Lem:Basic1}, we see that $\FF_0(x)=f(x)$ is monogenic. %, and by Proposition \ref{Prop:Yokoi}, $\alpha=(k+\sqrt{k^2+4})/2$ is the fundamental unit of $\Q(\alpha)=\Q(\sqrt{\D})$. %, where $\D$ is as defined in \eqref{Eq:D}.
   %Note that $\FF_0(\alpha)=0$.
 For $n\ge 0$, define
  \[\alpha_n:=\alpha^{1/s^n} \quad \mbox{and} \quad K_n:=\Q(\alpha_n).\]
   Then $\alpha_0=\alpha$ and, since $\FF_0(x)=f(x)$ is monogenic, we have that $\Delta(\FF_0)=\Delta(K_0)$. Additionally, for all $n\ge 1$, we have that   \[\FF_n(\alpha_n)=0\quad \mbox{and} \quad [K_{n}:K_{n-1}]=s,\] by Lemma \ref{Lem:Irreducible}. %\red{We show next that $\FF_1(x)$ is monogenic.}
 We assume that $\FF_{n-1}(x)$ is monogenic, so that $\Delta(\FF_{n-1})=\Delta(K_{n-1})$, and we proceed by induction on $n$ to show that $\FF_{n}(x)$ is monogenic. Let $\Z_{K_{n}}$ denote the ring of integers of $K_{n}$.  Consequently, by Theorem \ref{Thm:CD}, it follows that
 \begin{equation*}\label{Eq:CD}
 \Delta(\FF_{n-1})^s \mbox{ divides } \Delta(K_{n})=\dfrac{\Delta(\FF_{n})}{[\Z_{K_{n}}:\Z[\alpha_{n}]]^2},
 \end{equation*}
 which implies that
 \[[\Z_{K_{n}}:\Z[\alpha_{n}]]^2 \mbox{ divides }  \dfrac{\Delta(\FF_{n})}{\Delta(\FF_{n-1})^s}.\]
   We see from \eqref{Eq:Disc} that
\begin{align*}
 \abs{\Delta(\FF_{n-1})^s}&=(-b)^{s^{n}-s}s^{2(n-1)s^{n}}(a^2+4b)^{s^{n}}\quad \mbox{ and}\\ \abs{\Delta(\FF_{n})}&=(-b)^{s^{n}-1}s^{2ns^{n}}(a^2+4b)^{s^{n}}.
 \end{align*}
 Hence,
 \[\abs{\dfrac{\Delta(\FF_{n})}{\Delta(\FF_{n-1})^s}}=(-b)^{s-1}s^{2s^{n}}.\]  Thus, it is enough to show that $\gcd(bs,[\Z_{K_{n}}:\Z[\alpha_{n}]])=1$.
  Recall that $\gcd(b,s)=1$ by hypothesis.

  Suppose first that $p$ is a prime divisor of $b$.
  If $p\mid a$, then it follows that
  \begin{equation}\label{Eq:not zero}
  [\Z_{K_{n}}:\Z[\alpha_{n}]]\not \equiv 0\pmod{p}
  \end{equation} by item \ref{JKS:I1} of Theorem  \ref{Thm:JKS} since $b$ is squarefree.
  So, assume that $p\nmid a$.
  In this case, we apply item \ref{JKS:I3} of Theorem  \ref{Thm:JKS} to $\FF_{n}(x)$.  Observe that $A_1=0$ since $p\nmid s$, and $B_2=-b/p\not \equiv 0 \pmod{p}$ since $b$ is squarefree. Thus, the first condition of item \ref{JKS:I3} holds, and therefore once again we have \eqref{Eq:not zero}.

  Suppose next that $p$ is a prime divisor of $s$ with $s^n=p^mv$, where $m\ge 1$ and $p\nmid v$. Note that $p\nmid b$.

  We first address the prime $p=2$. If $2\mid a$, then we apply item \ref{JKS:I2} of Theorem  \ref{Thm:JKS} to $\FF_{n}(x)$. Observe from conditions $(*)$ that in this case, $2\mid \mid a$ and $b\equiv 1 \pmod{4}$ since respectively, $a\not \equiv 0 \pmod{4}$ and $\D=(a/2)^2+b$ is squarefree. Thus,
  \[A_2=-\frac{a}{2}\not \equiv 0 \pmod{2} \quad \mbox{and} \quad B_1=\frac{-b+b^{2^{m+1}}}{2}\equiv 0 \pmod{2},\] from which we conclude that the second condition of item \ref{JKS:I2} of Theorem  \ref{Thm:JKS} holds. Therefore, in this case, we have \eqref{Eq:not zero}. If $2\nmid a$, we apply item \ref{JKS:I4} of Theorem  \ref{Thm:JKS} to $\FF_n(x)$. Since $2$ is not an $(a,b)$-Wall-Sun-Sun prime, we have that
  \[(a,b)_4\in \{(1,1),(1,3),(3,1)\}\]
  from item \ref{I0:p=2} of Theorem \ref{Thm:Period}. Since
  \[c^{2^m}-c\equiv \left\{\begin{array}{cl}
  0 \pmod{4} & \mbox{if $c\equiv 1 \pmod{4}$}\\
  2 \pmod{4} & \mbox{if $c\equiv 3 \pmod{4}$,}
  \end{array}\right.\] for an integer $c$ and and an integer $m\ge 1$, it follows from \eqref{Eq:H} and \eqref{Eq:Binom} that
  \[\overline{G}(x)=x^{2v}+x^v+1 \quad \mbox{and}\]
  \[\overline{H}(x)=\left\{
  \begin{array}{ll}
   x^{2^{m-1}v} & \mbox{if $(a,b)_4=(1,1)$}\\
   \left(x^{v}+1\right)^{2^{m-1}} & \mbox{if $(a,b)_4=(1,3)$}\\
   x^{2^{m-1}v}\left(x^v+1\right)^{2^{m-1}} & \mbox{if $(a,b)_4=(3,1)$.}
   \end{array}\right.\] Hence, for every zero $\rho$ of $\overline{H}(x)$, we see that $\overline{G}(\rho)=1$. Thus, $G(x)$ and $H(x)$ are coprime modulo 2, so that \eqref{Eq:not zero} holds with $p=2$.

  Now suppose that $p\ge 3$. If $\delta_p=-1$, then \eqref{Eq:not zero} follows from Lemma \ref{Lem:Main1} and item \ref{I1:LemMain2} of Lemma \ref{Lem:Main2}, since $p$ is not an $(a,b)$-Wall-Sun-Sun prime. Recall that $\delta_3=-1$ by hypothesis. Then, finally, if $p\ge 5$ with $\delta_p=0$, it follows from item \ref{I2:LemMain2} of Lemma \ref{Lem:Main2} that \eqref{Eq:not zero} holds, completing the proof of the theorem.
  \end{proof}

%\section*{Acknowledgments}

%\section*{Data Availability Statement}
%The author confirms that all relevant data are included in the article.


\begin{thebibliography}{99}

\bibitem{Bouazzaoui1} Z. Bouazzaoui, \emph{Fibonacci numbers and real quadratic $p$-rational fields}, Period. Math. Hungar. {\bf 81} (2020), no. 1, 123--133.

\bibitem{Bouazzaoui2} Z. Bouazzaoui, \emph{On periods of Fibonacci sequences and real quadratic p-rational fields}, Fibonacci Quart. {\bf 58} (2020), no. 5, 103--110.

\bibitem{Cohen} H. Cohen, \emph{A Course in Computational Algebraic Number Theory}, {Springer-Verlag}, 2000.

\bibitem{CDP} R. Crandall, K. Dilcher and C. Pomerance, \emph{A search for Wieferich and Wilson primes},
Math. Comp. {\bf 66} (1997), no. 217, 433--449.

\bibitem{EJ} A.-S. Elsenhans and J. Jahnel, \emph{The Fibonacci sequence modulo $p^2$-An investigation by computer for $p < 1014$}, \url{arXiv:1006.0824v1}.

\bibitem{GRS} S. Gupta, P. Rockstroh and F. E. Su, \emph{Splitting fields and periods of Fibonacci sequences modulo primes},
Math. Mag. {\bf 85} (2012), no. 2, 130--135.

\bibitem{HK}  P. Haggard and J. Kiltinen, \emph{Binomial expansions modulo prime powers}, {Internat. J. Math. Math. Sci.} {\bf 3} (1980), no. 2, 397--400.

\bibitem{JKS2} A. Jakhar, S. Khanduja and N. Sangwan, \emph{Characterization of primes dividing the index of a trinomial}, Int. J. Number Theory {\bf 13} (2017), no. 10, 2505--2514.

\bibitem{JonesEJM} L. Jones, \emph{A Connection Between the Monogenicity of Certain Power-Compositional Trinomials and $k$-Wall-Sun-Sun Primes}, \url{http://arxiv.org/abs/2211.14834}.

\bibitem{HJBAMS} J. Harrington and L. Jones, \emph{A note on generalized Wall-Sun-Sun primes}, Bull. Aust. Math. Soc. (to appear).

%\bibitem{LJEisenstein} L. Jones, \emph{The monogenity of power-compositional Eisenstein polynomials}, Ann. Math. Inform. (2022) \url{https://ami.uni-eszterhazy.hu/index.php?vol=Latest}. %(to appear).

\bibitem{Neukirch} J. Neukirch, \emph{Algebraic Number Theory}, Springer-Verlag, Berlin, 1999.

\bibitem{Renault} M. Renault, \emph{The period, rank, and order of the $(a,b)$-Fibonacci sequence mod $m$}, Math. Mag. {\bf 86} (2013), no. 5, 372--380.

\bibitem{Robinson}  D. W. Robinson, \emph{A note on linear recurrent sequences modulo $m$}, Amer. Math. Monthly {\bf 73} (1966), 619--621.

\bibitem{S} A.~Schinzel, \emph{Polynomials with Special Regard to Reducibility}, Encyclopedia of Mathematics and its Applications, {\bf 77}, Cambridge University Press, Cambridge, 2000.

\bibitem{SunSun} Zhi Hong Sun and Zhi Wei Sun, \emph{Fibonacci numbers and Fermat's last theorem}, Acta Arith. {\bf 60} (1992), no. 4, 371--388.

\bibitem{Swan}  R. Swan, \emph{Factorization of polynomials over finite fields}, Pacific J. Math. {\bf 12} (1962), 1099--1106.

\bibitem{Wall}  D. D. Wall, \emph{Fibonacci series modulo $m$}, Amer. Math. Monthly {\bf 67} (1960), 525--532.

\bibitem{Wiki0} \emph{Fundamental Discriminant} \url{https://en.wikipedia.org/wiki/Fundamental_discriminant}

\bibitem{Wiki1} \emph{Wieferich Prime} \url{https://en.wikipedia.org/wiki/Wieferich_prime}

\bibitem{Wiki2} \emph{Wall-Sun-Sun Prime}\\ \url{https://en.wikipedia.org/wiki/Wall%E2%80%93Sun%E2%80%93Sun_prime}

\bibitem{Y}  H. Yokoi, \emph{On real quadratic fields containing units with norm $-1$}, Nagoya Math. J. {\bf 33} (1968), 139--152.
\end{thebibliography}
\end{document}